\documentclass [11pt,oneside]{amsart}

\reversemarginpar \textwidth=14cm \textheight=19cm

\usepackage{amssymb} \usepackage{amsfonts} \usepackage{amsmath}
\usepackage{amsthm} \usepackage{epsfig} 

\usepackage{caption}

\newtheorem{lemma}{Lemma}
\newtheorem{teo}[lemma]{Theorem}
\newtheorem{prop}[lemma]{Proposition}
\newtheorem{cor}[lemma]{Corollary}

\theoremstyle{definition}
\newtheorem{defn}[lemma]{Definition}

\newtheorem{rem}[lemma]{Remark} 

\newcommand{\sz}{\ensuremath {\mathrm{pr}}}
\newcommand{\D}{\ensuremath {\Delta}}

\newcommand{\matN}{\ensuremath {\mathbb{N}}}
\newcommand{\R} {\ensuremath {\mathbb{R}}}

\newcommand{\calC} {\ensuremath {\mathcal{C}}}
\newcommand{\calP} {\ensuremath {\mathcal{P}}}

\newcommand{\finedimo}{{\hfill\hbox{$\square$}\vspace{2pt}}}

\author{Roberto Frigerio}
\author{Alessandro Sisto}

\address{Dipartimento di Matematica \\
Universit\`a di Pisa \\
Largo B.~Pontecorvo 5 \\
56127 Pisa, Italy}

\address{Scuola Normale Superiore\\
piazza dei cavalieri 7 \\
56127 Pisa, Italy}
 
\email{frigerio@dm.unipi.it, a.sisto@sns.it}

\title[Characterizing hyperbolic spaces and real trees]{Characterizing hyperbolic spaces
and real trees}



\thanks{}

\begin{document}

\begin{abstract}
Let $X$ be a geodesic metric space. Gromov proved that there exists
$\varepsilon_0>0$ such that if every sufficiently large triangle $\D$
satisfies the Rips condition with constant $\varepsilon_0\cdot \sz(\D)$, where
$\sz(\D)$ is the perimeter $\D$, then $X$ is hyperbolic. We give 
an elementary proof of this fact, also giving an estimate for $\varepsilon_0$. 
We also show that if all the triangles $\D\subseteq X$
satisfy the Rips condition with constant $\varepsilon_0\cdot \sz(\D)$, 
then $X$ is a real tree.

Moreover, we point out how this characterization of hyperbolicity 
can be used to
improve a result by Bonk, and to provide  an easy proof
of the (well-known) fact 
that $X$ is hyperbolic if and only if every asymptotic cone of
$X$ is a real tree. 
\end{abstract}

\maketitle

\section{Preliminaries and statements}\label{construction:section}

Let $(X,d)$ be a metric space. A map $\gamma\colon [0,1]\to X$
is a \emph{geodesic} if there exists $k\geq 0$ such that $d(\gamma(t),\gamma(s))=k|t-s|$
for every $t,s\in [0,1]$. The space $X$ is \emph{geodesic} if any pair of points in $X$
can be connected by a geodesic, and \emph{uniquely geodesic} if such a geodesic is unique.
With an abuse, we identify geodesics and their
images, and we let $[x,y]$ denote a geodesic joining $x$ to $y$, even though this
geodesic is not unique.
A \emph{triangle} with vertices $x,y,z$ is the union of three geodesics
$[x,y],[y,z],[z,x]$, called \emph{sides}, and will be denoted by $\D(x,y,z)$.
We denote by $\sz (\D)$ the perimeter of $\D$, \emph{i.e.}~we set
$\sz(\D(x,y,z))=d(x,y)+d(y,z)+d(z,x)$.

\subsection{Gromov hyperbolic spaces and real trees}

For $A\subseteq X$ and $\varepsilon>0$, 
we set $N_\varepsilon (A)=\{x\in X:\, d(x,A)\leq \varepsilon\}$
A triangle with sides $\ell_1,\ell_2,\ell_3$ satisfies the Rips condition
with constant $\delta$ if for $\{i,j,k\}=\{1,2,3\}$ we have $\ell_i\subseteq N_{\delta} (\ell_j\cup \ell_k)$. 
A geodesic space $X$ is \emph{$\delta$-hyperbolic} if every triangle
in $X$ satisfies the Rips condition with constant $\delta$, and it is
\emph{hyperbolic} if it is $\delta$-hyperbolic for some $\delta\geq 0$.

A $0$-hyperbolic geodesic space is also called a \emph{real tree}. 
It is easily seen that a real tree is uniquely geodesic, and
that if $[x,y],[y,x]$ are geodesics in a real tree such that
$[x,y]\cap [y,z]=\{y\}$, then $[x,z]=[x,y]\cup [y,z]$.

\subsection{The main results}

Let $X$ be a fixed geodesic space. For every triangle $\D$ in $X$ we 
provide a measure of how much non-hyperbolic $\D$ is by setting
$$
\delta (\D) =  \inf  \{\delta:\,
\Delta\ {\rm satisfies\ the\ Rips\ condition\ with\ constant}\ \delta\}.
$$ 
Of course, 
for every $\D$ we have $4\delta (\D)\leq \sz(\D)$. 
 
Let
$\Omega_X \colon \R^+\to \R^+$ be defined as follows:
$$
\Omega_X (t)=\sup \{\delta (\D),\ \D\ {\rm triangle\ in}\ X\ {\rm with}\ \sz(\D)\leq t\}.
$$
By the very definition, $X$ is hyperbolic if and only if $\Omega_X$ is bounded.
Our main result, which will be proved in Section~\ref{main:sec}, is the following:

\begin{teo}\label{main:teo}
Let $X$ be a geodesic space. Then $X$ is hyperbolic if and only
if $$\limsup_{t\to\infty} \frac{\Omega_X (t)}{t}< \frac{1}{32}.$$
\end{teo}

Using tools from plane conformal geometry, Gromov proved in~\cite{gro1}
that a constant $\varepsilon_0>0$ exists
such that if $\limsup_{t\to\infty} \Omega (t)/t\leq \varepsilon_0$, then $X$ is hyperbolic.
Our proof of Theorem~\ref{main:teo} is completely elementary, and gives
for $\varepsilon_0$ the estimate of $1/32$.

Observe that by the very definitions we have 
$$ \sup_t \frac{\Omega_X (t)}{t}
=\sup \left\{ \frac{\delta(\D)}{\sz(\D)},\, \D\ {\rm triangle\ in} \ X\right\}.$$
The argument developed for proving Theorem~\ref{main:teo} 
also gives the following:

\begin{teo}\label{main3:teo}
Let $X$ be a geodesic space. Then $X$ is a real tree if and only if
$$\sup \left\{\frac{\delta(\D)}{\sz(\D)},\, \D\ {\rm triangle\ in}\ X\right\}
< \frac{1}{32}.$$
\end{teo}

Theorems~\ref{main:teo} and~\ref{main3:teo} will be proved in Section~\ref{main:sec}.

\subsection{Asymptotic cones}

In Section~\ref{cor:sec} we will show how Theorem~\ref{main:teo}
can be used to provide short proofs (and a slight improvement) of other known characterizations of 
hyperbolic spaces. In order to do this, we first need the definition
of \emph{asymptotic cone} of a metric space. Roughly speaking, 
the asymptotic cone of a metric space gives a picture of the metric
space as ``seen from infinitely far away''. It was introduced by Gromov in~\cite{gropol},
and formally defined in~\cite{wilkie}. 

A \emph{filter} on $\matN$ is a set $\omega\subseteq \calP(\matN)$ satisfying the following conditions:
\begin{enumerate}
\item
$\emptyset\notin \omega$;
\item
$A,B\in \omega\ \Longrightarrow\  A\cap B\in \omega$;
\item 
$A\in \omega,\ B\supseteq A\ \Longrightarrow\ B\in\omega$.
\end{enumerate}
For example, the set of complements of finite subsets of $\matN$
is a filter on $\matN$, known as the \emph{Fr\'echet filter} on $\matN$.

A filter $\omega$ is a \emph{ultrafilter} if for every $A\subseteq\matN$
we have either $A\in\omega$ or $A^c\in\omega$, where $A^c :=\matN\setminus A$.
An ultrafilter is \emph{non-principal} if it does not contain any finite subset
of $\matN$. 

It is readily seen that a filter is a ultrafilter if and only if it is maximal
with respect to inclusion. Moreover, an easy application of Zorn's Lemma
shows that any filter is contained in a maximal one. Thus, non-principal
ultrafilters exist (just take any maximal filter
containing the Fr\'echet filter). 

Let a non-principal ultrafilter $\omega$ on $\matN$ be fixed from now on.
If $X$ is a topological space, and $(x_n)\subseteq X$ is a  
sequence in $X$, 
we say that $\omega - \lim x_n =x_\infty$ if 
for every neughbourhood $U$ of $x_\infty$ the set
$\{n\in\matN:\, x_n\in U\}$ belongs to $\omega$.
It is easily seen that if $X$ is Hausdorff then the $\omega$-limit above,
if it exists, is unique. Moreover, any sequence in any compact space admits
a $\omega$-limit. For example, any sequence $(a_n)$ in
$[0,+\infty]$ admits a unique $\omega$-limit.

Now let $(X,d)$ be a metric space, $(x_n)\subseteq X$ be a sequence of
base-points, and $(d_n)\subset \R^+$ a sequence of rescaling factors
diverging to infinity. Let $\calC$ be the set of sequences $(y_n)\subseteq X$
such that $\omega - \lim d(x_n,y_n)/d_n<+\infty$, and consider the equivalence
relation defined on $\calC$ as follows:
$$
(y_n)\sim (z_n)\quad \Longleftrightarrow \quad
\omega-\lim \frac{d(y_n,z_n)}{d_n}=0.
$$
We set $X_\omega ((x_n),(d_n))=\calC/\sim$, end endow it with the well-defined
distance $d_\omega$ such that $$d_\omega ([(y_n)],[(z_n)])=\omega-\lim \frac{d(y_n,z_n)}
{d_n}.$$ 

\begin{defn}\label{cone:def}
The metric space $(X_\omega((x_n),(d_n)),d_\omega)$ is the \emph{asymptotic cone}
of $X$ with respect to the ultrafilter $\omega$, 
the basepoints $(x_n)$ and the rescaling factors $(d_n)$.
\end{defn}

As stated in~\cite{gro1,gro2}, a space $X$ is hyperbolic if and only if
every asymptotic cone of $X$ is a real tree (see~\cite{dru1} for an elementary proof). 
We will show in Section~\ref{cor:sec}
how Theorem~\ref{main:teo} easily implies this characterization of hyperbolicity
(see Proposition~\ref{main3:prop}).

\subsection{Detours}

The notion of \emph{detour} we are now going to recall was introduced by Bonk in~\cite{bonk},
where a characterization of hyperbolicity was given in terms of detour growth (see~Theorem~\ref{bonk:teo}).
Let $(X,d)$ be a geodesic space and let $t>0$. A \emph{$t$-detour} is a continuous map $\gamma:[0,1]\to X$
such that there exist a geodesic $[\gamma(0),\gamma(1)]$ and a point $z\in [\gamma(0),\gamma(1)]$
such that $d(x,{\rm Im}\, \gamma)\geq t$. The detour growth function $G_X\colon (0,\infty)\to (0,\infty]$
is defined as follows:
$$
G_X (t)=\inf \{ {\rm lenght} (\gamma):\, \gamma\ {\rm is\ a}\ t-{\rm detour} \}.
$$
Note that $G_X (t)=\infty$ if and only if there exist no rectifiable $t$-detours in $X$,
\emph{e.~g.}~if $X$ is a real tree (see Lemma~\ref{tree:lemma}).
The following result is proved in~\cite{bonk}:

\begin{teo}[Bonk]\label{bonk:teo}
A geodesic space $X$ is hyperbolic if and only if 
$$
\lim_{t\to\infty} \frac{G_X (t)}{t}=+\infty.
$$
\end{teo}

Using Theorem~\ref{main:teo}, in Section~\ref{cor:sec} we prove the following:

\begin{teo}\label{main2:teo}
A geodesic space $X$ is hyperbolic if and only if 
$$
\liminf_{t\to\infty} \frac{G_X (t)}{t} > 30.
$$  
\end{teo}

\subsection{Looking for optimal constants}
A very natural problem is to compute (or to give better estimates on)
the largest constants 
which could replace $1/32$ in the statements of Theorems~\ref{main:teo},~\ref{main3:teo}.
By Theorem~\ref{main:teo}, the set 
$$\{\varepsilon>0:\, 
{\rm every\ geodesic\ space}\
X\ {\rm with}\ \limsup_{t\to\infty} \frac{\Omega_X (t)}{t} <\varepsilon\ {\rm is\ hyperbolic}\}$$
is non-empty. Being bounded, such set admits a lowest upper bound, which
is readily seen to be a maximum, and will be denoted by
$\varepsilon_H$. In the same way, it makes sense to define
$\varepsilon_T$ as the largest constant such that every geodesic space
$X$ with $\sup_t \Omega_X (t)/t <\varepsilon_T$ is a real tree.

The following proposition is proved in Section~\ref{final:sec}, and provides
an upper bound for $\varepsilon_H,\varepsilon_T$:

\begin{prop}\label{eucli:prop}
For every $t>0$ we have
$$
\Omega_{\R^2} (t)=\frac{1}{2}\cdot \left(\frac{\sqrt{5}-1}{2}\right)^{\frac{5}{2}}\cdot t
\approx 0.15\cdot t.
$$
\end{prop}

From now on, we set $\eta_0=(\sqrt{5}-1)^{5/2}/2^{7/2}$. 
Since $\R^2$ is not hyperbolic, we have the following:

\begin{cor}\label{est:cor}
The following inequalities hold:
$$\frac{1}{32}\leq \varepsilon_H\leq\eta_0,\quad \frac{1}{32}\leq \varepsilon_H\leq\eta_0.$$
\end{cor}

Our proof of Theorem~\ref{main:teo} was intended to give a somewhat 
significant estimate of $\varepsilon_H,\varepsilon_T$ (in fact,
similar but shorter arguments can be provided 
in order to show just that $\varepsilon_H>0$, $\varepsilon_T>0$ exist).
However, there are no reasons why $1/32$ should provide a good approximation  
of $\varepsilon_H$ and
$\varepsilon_T$. On the other hand, a recent result by Wenger~\cite{wenger}
on the sharp isoperimetric constant for hyperbolic spaces 
seems to suggest that the Euclidean plane could provide sharp bounds
on the behaviour of curves and triangles in hyperbolic spaces, so that
$\varepsilon_H$ (and $\varepsilon_T$, see Proposition~\ref{est2:prop}) could be
not too far from $\eta_0$.
Finally, it seems quite reasonable that
$\varepsilon_H=\varepsilon_T$, but at the moment we are just able to prove the following:

\begin{prop}\label{est2:prop}
$\varepsilon_H\leq\varepsilon_T$.
\end{prop}

The proof of Proposition~\ref{est2:prop} is independent from that of Theorem~\ref{main3:teo},
so we get Theorem~\ref{main3:teo} also as a corollary of Theorem~\ref{main:teo}
and Proposition~\ref{est2:prop}.

\section{The main argument}\label{main:sec}

This section is devoted to the proofs of Theorems~\ref{main:teo},~\ref{main3:teo}.
Let $X$ be a fixed geodesic space. 
In what follows, every time two points $x,y$ belong to a given geodesic 
$\ell$, we denote by $[x,y]$ the (unique) geodesic
joining $x$ to $y$ such that $[x,y]\subseteq \ell$. In that case,
we also suppose that the symbol $\D(x,y,z)$ denotes
a triangle $[x,y]\cup [y,z]\cup [z,x]$ such that $[x,y]\subseteq \ell$.
We begin with the following:

\begin{lemma}\label{lemma}
Let $\rho,\alpha>0$ and let $\Delta\subseteq X$ be a geodesic triangle with sides $l_1,l_2,l_3$ such that 
${\rm length}(l_1)\leq \alpha \rho$ and $\delta(\Delta)\leq \rho+1$. 
Then for each $p\in l_1$ we have 
$$d(p,l_2 \cup l_3)\leq  \Omega_X ((4\alpha+4)\rho +6)+ \Omega_X ((2\alpha+4)\rho+6).$$
\end{lemma}

\begin{proof}
If ${\rm length} (l_i)\leq (\alpha +1)\rho +2$ for $i=2,3$, 
then $\sz (\D)\leq (3\alpha +2) \rho +4$, whence the conclusion
since $\Omega_X$ is an increasing function. So, 
if $a_i\in\D$ is the vertex opposite to the side $l_i$, up to
exchanging $l_2$ with $l_3$ we can 
take $q\in l_2$ such that $d(a_3,q)=(\alpha+1)\rho +2$.
Since ${\rm length} (l_1)\leq \alpha \rho$ we get 
$d(q,l_1)\geq \rho +2$,  
so $\delta (\D)\leq \rho +1$ implies that $r\in l_3$ exists such that 
$d(q,r)\leq \rho +1$. 
Since 
$$
\begin{array}{lllll}
d(a_3,r)&\leq& d(a_3,q)+d(q,r) &\leq & (\alpha+2)\rho +3,\\
d(a_2,r)&\leq& d(a_2,a_3)+d(a_3,r) &\leq& 2(\alpha+1)\rho +3,
\end{array}
$$ 
setting $\D_1 =l_1 \cup [a_3,r]\cup [r,a_2]$,
$\D_2=[a_3,q]\cup [q,r]\cup [r,a_3]$ 
we get
\begin{equation}\label{first:eq}
\sz(\D_1)\leq (4\alpha+4)\rho+6,\quad
\sz(\D_2)\leq (2\alpha+4)\rho+6.
\end{equation} 
Let now $p$ be any point of $l_1$, and consider the triangle 
$\D_1$. By~(\ref{first:eq}) there exists
$s\in [a_3,r]\cup [r,a_2]$ such that $d(p,s)\leq \delta(\D_1)\leq \Omega_X ((4\alpha+4)\rho +6)$.
If $s$ belongs to $[r,a_2]$, we are done.
Otherwise $s$ belongs to $[a_3,r]$, so
a point $t\in [a_3,q]\cup [q,r]$ exists such that 
$d(s,t)\leq \delta (\D_2)\leq \Omega_X ((2\alpha+4)\rho+6)$.
Thus $d(p,t)\leq \Omega_X ((4\alpha+4)\rho +6)+ \Omega_X ((2\alpha+4)\rho+6)$,
and if $t\in [a_3,q]$ we are done. Otherwise,
we have $t\in [q,r]$, so $d(p,q)\leq d(p,t)+d(t,q)\leq \Omega_X ((4\alpha+4)\rho +6)+\rho+1$.
Thus
$$
(\alpha+1)\rho+1=d(a_3,q)\leq d(a_3,p)+d(p,q)\leq d(a_3,p)+\Omega_X ((4\alpha+4)\rho +6)+\rho+1,
$$
whence $d(a_3,p)\geq \alpha\rho-\Omega_X ((4\alpha+4)\rho +6)$. This readily implies
$d(p,a_2)\leq \Omega_X ((4\alpha+4)\rho +6)$, whence $d(p,l_2\cup l_3)\leq 
\Omega_X ((4\alpha+4)\rho +6)$, and the conclusion at once.
\end{proof}

\begin{figure}
\begin{center}
\input{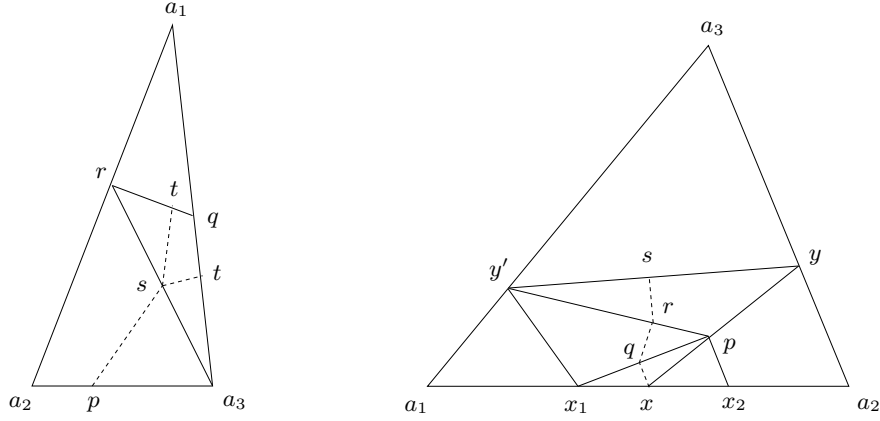} 
\caption{Notations for the proofs of Lemma~\ref{lemma} and Theorem~\ref{main:teo}.}
\label{prima:fig}
\end{center}
\end{figure}

\vspace{1pt}\noindent{\it Proof of Theorem~\ref{main:teo}.}
By contradiction, suppose $\Omega_X$ diverges. We set
$$\beta=\frac{1}{32} -\limsup_{t\to\infty} \frac{\Omega_X (t)}{t} >0,
\quad \kappa=\frac{1}{32}-\frac{\beta}{2}.$$
Let $\mu>0$ be large enough so that $\Omega_X (\mu)>(1/\beta)+1$ and $\Omega_X (l)\leq \kappa l$ for every 
$l\geq \Omega_X (\mu)-1$. 
Let $\Delta=\D(a_1,a_2,a_3)$ 
be a geodesic triangle with $\sz (\D)\leq \mu$ and $\lambda=\delta(\Delta)\geq \Omega_X (\mu)-1$. 
Up to reordering the vertices of $\Delta$, we may suppose there exist
$x\in [a_1,a_2]$ and $y\in [a_2,a_3]$ such that $d(x,[a_2,a_3]\cup [a_3,a_1])=d(x,y)=\lambda$. 
For $i=1,2$, let $x_i$ be the point on $[x,a_i]$ such that
$d(x,x_i)=\lambda/3$,
and let $p\in [x,y]$ be the point such that $d(x,p)=\lambda/3$.

Since $\sz(\D(x_1,x_2,p))\leq 2\lambda$, a point 
$q\in [x_1,p]\cup [x_2,p]$ exists such that $d(x,q)\leq
\Omega_X (2\lambda)\leq 2\kappa \lambda$. 
Without
loss of generality, we may suppose $q\in [x_1,p]$ 
(the following proof working exactly in the same way
also in the case $q\in [x_2,p]$).

Let $y'\in [a_1,a_3]\cup[a_2,a_3]$ be such that
$d(x_1,y')\leq \lambda=\delta (\D)$.
Since $\sz(\D (p,x_1,y'))\leq 2(d(p,x_1)+d(x_1,y'))\leq
(10/3)\lambda$, 
a point $r\in [x_1,y']\cup [p,y']$ exists such that
$d(q,r)\leq (10/3)\kappa\lambda$. Thus
\begin{equation}\label{stima1}
d(x,r)\leq d(x,q)+d(q,r)\leq \frac{16\kappa\lambda}{3}. 
\end{equation} 

Suppose $r\in [x_1,y']$. 
Then
$d(x_1, r)\geq d(x_1,x)-d(x,r)\geq ((1-16\kappa)/3)\lambda$.
On the other hand, since $d(x,[a_1,a_3]\cup [a_2,a_3])=\lambda$
and $\kappa <1/32$ we have
$$
\begin{array}{ccccccl}
\lambda & \leq & d(x,y')&\leq & d(x,r)+d(r,y')&=& d(x,r)+d(x_1,y')- d(x_1,r)\\
 & \leq &
(\frac{16\kappa}{3}+1-\frac{1-16\kappa}{3})\lambda & = & \frac{2+32\kappa}{3}\lambda & < &\lambda,
\end{array}
$$
a contradiction.
Thus $r\in [p,y']$.

Now $d(y,p)+d(p,y')\leq d(y,x)+d(x,x_1)+d(x_1,y')\leq (7/3)\lambda$,
so $\sz(\D(p,y',y))\leq (14/3)\lambda$, and a point
$s\in [y,y']\cup [p,y]$ exists such that 
$d(r,s)\leq (14/3)\kappa\lambda$.
By~(\ref{stima1}), 
it follows that
\begin{equation}\label{stima2}
d(x,s)\leq d(x,r)+d(r,s)\leq 10\kappa\lambda.
\end{equation}
Since $(10/32)\lambda< (1/3)\lambda = d(x,p)$, this implies
$s\in [y,y']$. Observe also that $y'\in [a_1,a_3]$,
because otherwise we would have $[y,y']\subseteq [a_2,a_3]$, and
$d(x, [a_2,a_3])\leq d(x,s)<\lambda$, a contradiction.

Consider now the triangle $\D(y,y',a_3)$. Of course $\sz (\D(y,y',a_3))\leq
\sz (\D)$, so 
$\delta (\D(y,y',a_3))\leq \Omega_X (\mu)\leq \lambda +1$. Since
$d(y,y')\leq d(y,x)+d(x,x_1)+d(x_1,y')\leq (7/3)\lambda$,
by Lemma~\ref{lemma}
we obtain
$$
\begin{array}{llll}
& d(s, [a_1,a_3]\cup [a_2,a_3]) & \leq & d(s,[y',a_3]\cup [y,a_3])\\
\leq & \Omega_X ((40/3)\lambda + 6)+\Omega_X ((26/3)\lambda +6) & \leq &
\kappa (22\lambda + 12).
\end{array}
$$
By~(\ref{stima2}), since $\kappa=1/32 -\beta/2$ and $\lambda>1/\beta$
we finally get $$d(x,[a_1,a_3]\cup [a_2,a_3])\leq d(x,s)+d(s, [a_1,a_3]\cup [a_2,a_3])
<32\kappa\lambda +12\kappa <\lambda,$$
a contradiction.
\finedimo

\vspace{1pt}\noindent{\it Proof of Theorem~\ref{main3:teo}.}
Let $X$ be a geodesic space such that
$\sup_{t} \Omega_X (t)/t<1/32$ and suppose by contradiction
that there exists $\mu>0$ with $\Omega_X (\mu)>0$. As in the proof
of Theorem~\ref{main:teo}, set 
$$ 
\beta=\frac{1}{32} -\sup_{t} \frac{\Omega_X (t)}{t} >0,
\quad \kappa=\frac{1}{32}-\frac{\beta}{2}.$$
Observe that a rescaling of the metric of $X$ does not affect
the hypothesis and the thesis of the theorem, so
we can assume
$\Omega_X (\mu)>(1/\beta)+1$.
Then a triangle
$\Delta\subseteq X$ exists such that 
$\sz (\D)\leq \mu$ and $\lambda=\delta(\Delta)\geq \Omega_X (\mu)-1$,
and the very same argument of the proof of Theorem~\ref{main:teo}
leads to a contradiction.

\section{Characterizing hyperbolic spaces}\label{cor:sec}

This section is devoted to the proof of the following result, which will
in turn imply Theorem~\ref{main2:teo}.

\begin{prop}\label{main3:prop}
Let $(X,d)$ be a geodesic space.
The following facts are equivalent:
\begin{enumerate}
\item
$X$ is hyperbolic;
\item
for any choice of a ultrafilter $\omega$, a sequence of basepoints $(x_n)\subseteq X$ and 
a sequence of rescaling factors $(d_n)\subseteq \R$,
the asymptotic cone $X_\omega ((x_n),(d_n))$ is a real tree;
\item
$\liminf_{t\to\infty} G_X(t)/t> 30$;
\item
$\limsup_{t\to\infty} \Omega_X(t)/t< 1/32$.
\end{enumerate}
\end{prop}

We show first an easy (and well-known) result
which will be needed in the proof of Proposition~\ref{main3:prop}:

\begin{lemma}\label{tree:lemma}
Suppose $(X,d)$ is a real tree and let $\gamma\colon [0,1]\to X$
be a continuous path with $\gamma(0)=x$, $\gamma (1)=y$. Then $[x,y]\subseteq {\rm Im}\, \gamma$.
\end{lemma}
\begin{proof}
Let $z\in X\setminus [x,y]$ and observe that since $[x,y]$ is compact a point $t\in [x,y]$ 
exists such that
$d(z,t)=d(z,[x,y])=k>0$. We claim that if $d(z',z)<k/2$ and 
$d(z',t')=d(z',[x,y])$, then $t=t'$. 
In fact, of course $[z',t']\cap [x,y]=\{t'\}$, 
so $[z',t]=[z',t']\cup [t',t]$. But $X$ being $0$-hyperbolic,
this implies $t'\in [t,z]\cup [z,z']$. Since $d(z',t')\geq d(z,t')- d(z,z')> k-k/2=k/2$,
we cannot have $t'\in [z,z']$, so $t' \in [t,z]$, whence $t'=t$ since
$[t,z]\cap [x,y]=\{t\}$, and the claim is proved.
This readily implies that the map $\pi\colon X\to [x,y]$
which sends $p\in X$ to its closest point  $\pi (p)\in [x,y]$ is well-defined, continuous and locally
constant on $X\setminus [x,y]$.

Being connected and containing $x,y$, the set
${\rm Im}\, (\pi\circ\gamma)\subseteq [x,y]$ equals in fact $[x,y]$.
So, suppose there exists $s\in [x,y]\setminus {\rm Im}\,\gamma$, and observe
that of course $s\neq x$. 
Then $(\pi\circ\gamma)^{-1} (s)\subseteq [0,1]$ is non-empty, closed and
open (because $\pi$ is locally constant on $X\setminus [x,y]$),
whence equal to $[0,1]$,
a contradiction since $\pi (\gamma (0))=x\neq s$.
\end{proof}

\begin{proof}
$(1)\Rightarrow (2).$
This implication is well-known, we sketch a proof of it
for the sake of completeness.
Suppose $(X,d)$ is $\delta$-hyperbolic. Then $(X,d/d_n)$ is obviously 
$(\delta/d_n)$-hyperbolic.

We first show that
$X_\omega:=((X_\omega, (x_n),(d_n)),d_\omega)$ is uniquely geodesic. So, let $[(y_n)],[(z_n)]\in X_\omega$,
and let $\gamma_n\colon [0,1]\to X$ be a geodesic joining $y_n$ to $z_n$ for every
$n\in\matN$. It is easily seen that the map $\gamma_\omega\colon [0,1]\to X_\omega$
defined by $\gamma (t)=[(\gamma_n (t))]$ is a geodesic.
Let $\psi\colon [0,1]\to X_\omega$ be 
a geodesic with the same endpoints as $\gamma$ and take $t_0\in [0,1]$.
If $\psi(t_0)=[(p_n)]$, let us consider a triangle $\D_n=[y_n,p_n]\cup [p_n,z_n]\cup
{\rm Im}\, \gamma_n \subseteq X$:
by $\delta$-hyperbolicity of $X$, a point
$q_n\in [y_n,p_n]\cup [p_n,z_n]$ exists such that $d(\gamma_n (t_0),q_n)\leq \delta$.
Of course, this implies $[(q_n)]=\gamma_\omega (t_0)$. In particular,
we have $d_\omega ([(q_n)],[(y_n)])=d_\omega (\gamma_\omega (t_0), [(y_n)])=d_\omega (\psi (t_0),[(y_n)])$
and $d_\omega ([(q_n)],[(z_n)])=d_\omega (\gamma_\omega (t_0), [(z_n)])=d_\omega (\psi (t_0),[(z_n)])$.
Since $q_n\in [y_n,p_n]\cup [p_n,z_n]$, this easily implies that $[(q_n)]=[(p_n)]$,
whence $\psi (t_0)=\gamma_\omega (t_0)$, and $\psi=\gamma_\omega$.

Let now $\D_\omega=[x^1_\omega,x^2_\omega]\cup [x^2_\omega,x^3_\omega]\cup [x^3_\omega,x^1_\omega]\subseteq
X_\omega$ be a geodesic triangle. We have just proved that, $X_\omega$ being uniquely geodesic,
$\D_\omega$ is in an obvious sense the $\omega$-limit of triangles $\D_n=[x^1_n,x^2_n]\cup
[x^2_n,x^3_n]\cup [x^3_n,x^1_n]$ such that $x^i_\omega=[(x^i_n)]$. With respect to the
rescaled metric $d/d_n$, these triangles satisfy
the Rips condition with constant $\delta/d_n$. Since 
$\lim_{n\to\infty} \delta/d_n =0$, this readily implies that $\D_\omega$ satisfies
the Rips condition with constant $0$, whence the conclusion.

\vspace{.3cm}\par

$(2)\Rightarrow (3).$
Arguing by contradiction, we will prove the stronger fact that, if $(2)$ holds, 
then $\lim_{t\to\infty} G_X (t)/t=
+\infty$. So, suppose there exist a constant $M>0$ and a diverging sequence $(t_n)\subseteq \R^+$ 
such that $G(t_n)/t_n < M$ for every $n\in\matN$. By the very definition of $G_X$,
for every $n\in\matN$ 
there exist points $x_n,y_n\in X$,
a path $\gamma_n \colon [0,1]\to X$ with $\gamma_n (0)=x_n$, $\gamma_n (1)=y_n$ and 
${\rm lenght}(\gamma_n)\leq M t_n$,
a geodesic $[x_n,y_n]$ and a point $z_n\in [x_n,y_n]$ such that $d(z_n,{\rm Im}\, \gamma_n) \geq t_n$.
Let now $\omega$ be any non-principal ultrafilter, and consider the asymptotic cone
$X_\omega:=(X_\omega ((x_n),(t_n)),d_\omega)$.

Since $d(x_n,y_n)\leq {\rm lenght}(\gamma_n)\leq M t_n$, as in the proof of $(1)\Rightarrow (2)$
one can prove that the $\omega$-limit of the geodesics $[x_n,y_n]$
defines a geodesic in $X_\omega$ joining $x_\omega :=[(x_n)]$ and $y_\omega :=[(y_n)]$. 
We denote
such a geodesic by $[x_\omega,y_\omega]$, and observe that $[(z_n)]\in 
[x_\omega,y_\omega]$. Without loss of generality, we may suppose
$\gamma_n$ is parameterized at constant speed. Since ${\rm lenght} (\gamma_n)\leq
Mt_n$, this implies that $\gamma_n$ is $Mt_n$-Lipschitz with respect to $d$,
whence $M$-Lipschitz with respect to the rescaled metric $d/t_n$.
It is readily seen that under this condition the map $\gamma_\omega\colon [0,1]\to X_\omega$
defined by $\gamma_\omega (t)=[(\gamma_n (t))]$ is a well-defined $M$-Lipschitz 
(whence continuous) arc. Moreover, since $d(z_n,{\rm Im}\, \gamma_n )\geq t_n$,
we have $[(z_n)]\notin {\rm Im}\, \gamma_\omega$. By Lemma~\ref{tree:lemma},
$X_\omega$ is not a real tree, a contradiction.

\vspace{.3cm}\par

$(3)\Rightarrow (4).$
Let $\ell_1,\ell_2,\ell_3$ be the edges of a geodesic triangle 
$\Delta\subseteq X$ and suppose
$\delta(\D)=d(p,\ell_2\cup \ell_3)$, where $p$ is a point of $\ell_1$.
Since ${\rm length}\, l_1\geq 2\delta (\D)$, 
a suitable parameterization of $\ell_2\cup \ell_3$ provides a 
$\delta(\D)$ detour of length at most $\sz (\D)-2\delta (\D)$. This implies that
for every $t\in \R^+$ and $\varepsilon>0$  we have $$G_X(\Omega_X (t)-\varepsilon)\leq t- 
2(\Omega_X (t) - \varepsilon).$$  

If $\Omega$ is bounded, there is nothing to prove, so, since
$\Omega_X$ is increasing, we may assume 
$\lim_{t\to\infty} \Omega (t)=+\infty$. 
Suppose now $\liminf_{t\to\infty} G_X(t)/t=\alpha>30$
and take $0<\varepsilon< (\alpha-30)/3$. Then
for $t$ sufficiently
large we have 
\begin{equation}\label{sti1:eq}
\frac{\varepsilon}{t}<\frac{1}{\alpha+2-2\varepsilon}-\frac{1}{\alpha+2-\varepsilon}
\end{equation}
and
\begin{equation}\label{sti2:eq}
t -2(\Omega_X (t)- \varepsilon)\geq G_X(\Omega_X (t)-\varepsilon)> 
(\alpha-\varepsilon) (\Omega_X (t)-\varepsilon).
\end{equation}
By~(\ref{sti2:eq}) we get
$(\Omega_X (t)-\varepsilon)/t< 1/(\alpha+2-\varepsilon)$, whence,
by~(\ref{sti1:eq}), $\Omega_X (t)/t< 1/(\alpha+2-2\varepsilon)$. 
Thus $\limsup_{t\to\infty} \Omega_X (t)/t\leq 1/(\alpha+2-2\varepsilon)<1/32$.
\vspace{.3cm}\par

$(4) \Rightarrow (1)$ is just the result proved in Theorem~\ref{main:teo}.
\end{proof}

\section{The Euclidean case}\label{final:sec}

This section is devoted to the proof of Proposition~\ref{eucli:prop}.
In what follows, for every $A,B\in\R^2$ we will denote by $\overline{A B}$ 
the distance $d(A,B)$. The following lemma readily implies $\Omega_{\R^2}(1)\geq \eta_0$.

\begin{lemma}\label{isoscele:lemma}
Let $\D=\D(B_1,B_2,B_3)\subset\R^2$ be a triangle with $\sz(\D)\leq 1$ and 
$\widehat{B_3 B_1 B_2}=\widehat{B_1 B_2 B_3}=\alpha$, and let $Q$ be the midpoint
of $[B_1,B_2]$. Then
$d(Q,[B_1,B_3]\cup [B_2,B_3])\leq \eta_0$, the equality holding if and only if
$\sz(\D)=1$ and $\cos \alpha=(\sqrt{5}-1)/2$.
\end{lemma}
\begin{proof}
It is easily seen that $d(Q,[B_1,B_3]\cup [B_2,B_3])=(\overline{B_1 B_2}\sin\alpha)/2$,
while $\sz(\D)=\overline{B_1 B_2}(1+\cos\alpha)/(\cos\alpha)$. Let $\alpha_0\in (0,\pi/2)$ be such
that $\cos\alpha_0= (\sqrt{5}-1)/2$. An easy computation shows that 
for every $\alpha\in (0,\pi/2)$ we have 
$$
\frac{\delta(\D)}{\sz (\D)}=\frac{\sin\alpha\cos\alpha}{2(1+\cos\alpha)}\leq 
\frac{\sin\alpha_0 \cos\alpha_0}{2(1+\cos\alpha_0)}=\eta_0,
$$
the equality holding if and only if $\alpha=\alpha_0$, whence the conclusion. 
\end{proof}

\vspace{1pt}\noindent{\it Proof of Theorem~\ref{main:teo}.}
It will be sufficient to show that $\Omega_{\R^2} (1)=\eta_0$:
in fact, 
any rescaling of $\R^2$ is isometric to $\R^2$ itself,
so for any $t>0$ we obviously have $\Omega_{\R^2} (t)=
\Omega_{\R^2} (1)\cdot t$.

Let $\D=\D(A_1,A_2,A_3)\subseteq\R^2$ be a triangle with $\sz(\D)\leq 1$. 
Up to reordering
$A_1,A_2,A_3$, we may suppose that $P\in [A_1,A_2]$ exists such that
$\delta(\D)=\{d(P,[A_1,A_3]\cup [A_2,A_3])=d(P,[A_2,A_3])$.

\begin{figure}
\begin{center}
\input{seconda.pstex_t} 
\caption{Computing $\Omega_{\R^2} (1)$: the case when $\widehat{A_1 A_2 A_3}\geq \pi/2$.}
\label{seconda:fig}
\end{center}
\end{figure}

If $ \widehat{A_1 A_2 A_3} \geq \pi/2$, then take 
$A'_3\in [A_1,A_3]$ in such a way that
$\widehat{A_1 A_2 A'_3}=\pi/2$, set $\D'=\D(A_1,A_2,A'_3)$ and let $P'\in [A_1,A_2]$
be the farthest point from $[A_1,A'_3]\cup [A_2,A'_3]$. 
Of course we have $d(P',[A_1,A'_3]\cup [A_2,A'_3])\geq \delta (\D)$
and $\sz (\D')\leq \sz (\D)$. 
Let now $\ell$ be the line passing through $A'_3$ which is parallel 
to $[A_1,A_2]$, take $A''_3\in\ell$ in such a way that
$\overline{A_1 A''_3}=\overline{A_2 A''_3}$ and set
$\D''=\D(A_1,A_2,A''_3)$. An easy computation
shows that if $P''$ is the midpoint of $[A_1,A_2]$, then
$d(P'',[A_1,A''_3]\cup [A_2,A''_3])\geq d(P',[A_1,A'_3]\cup [A_2,A'_3])$,
while $\sz(\D'')\leq \sz (\D')$. Since $\sz(\D'')\leq \sz(\D)\leq 1$
and $d(P'',[A_1,A''_3]\cup [A_2,A''_3])\geq \delta (\D)$, 
by Lemma~\ref{isoscele:lemma} we have $\delta(\D)\leq \eta_0$.

\begin{figure}
\begin{center}
\input{terza.pstex_t} 
\caption{Computing $\Omega_{\R^2} (1)$: the case when $\widehat{A_1 A_2 A_3}\leq \pi/2$.}
\label{terza:fig}
\end{center}
\end{figure}

Suppose now $ \widehat{A_1 A_2 A_3}, \widehat {A_2 A_1 A_3} \leq \pi/2$, and let
$\ell_i$ be the half-line with endpoint $A_3$ containing $A_i$. It is easily seen
that $\delta(\D)=d(P,\ell_1)=d(P,\ell_2)$. Let now $r$ be the line
orthogonal to $[A_3,P]$ and passing through $P$, and set
$A'_i=\ell_i\cap r$, $\D'=\D(A'_1,A'_2,A_3)$. Of course $\overline{A'_1 A_3}=\overline{A'_2 A_3}$
and $d(p',[A'_1,A_3]\cup [A'_2,A_3])=\delta(\D)$, while an easy computation shows that
$\sz (\D')\leq \sz(\D)\leq 1$. As before, Lemma~\ref{isoscele:lemma} now 
implies  $\delta(\D)\leq \delta (\D')\leq \eta_0$.

We have thus proved that if $\D\subset \R^2$ is a triangle with $\sz(\D)\leq 1$, 
then $\delta(\D)\leq \eta_0$. This implies 
$\Omega_X (1)\leq \eta_0$, whence the conclusion.
\finedimo

\section{Some remarks on the optimal constants}\label{final2:sec}

This section is entirely devoted to the proof of
Proposition~\ref{est2:prop}. We will show that, if
$(X,d)$ be is geodesic space such that
$$
\sup_t \frac{\Omega_X (t)}{t}=\alpha < \varepsilon_H,
$$
then $(X,d)$ is a real tree.
The idea of the proof is as follows: we
realize $X$ as an isometrically embedded
subspace of the asymptotic cone of a suitable geodesic space $Y$, 
chosen in such a way that $\limsup_{t\to \infty}
\Omega_Y (t)/t<\varepsilon_H$. This ensures that $Y$ is hyperbolic,
which in turn implies that $X$ is a real tree.

So, let $p\in X$ be a fixed basepoint, and  
let $Y\subseteq X\times\R$ be defined as follows:
$$
Y=\left(\{p\}\times\R\right) \cup \left(\bigcup_{i\in\matN} X\times \{i\}\right).
$$
We define a distance $\widetilde{d}$ on $Y$ 
by setting:
$$
\widetilde{d} ((x,t),(x',t'))=
\left\{
\begin{array}{lll}
i\cdot d(x,p) + j\cdot d(p,x') + |t-t'| & {\rm if} & t\neq t'\\
i\cdot d(x,x') & {\rm if} & t=t'
\end{array}
\right.
$$
It is easily seen that $(Y,\widetilde{d})$ is a geodesic metric space,
and that in $(Y,\widetilde{d})$ there are not unexpected geodesics. More precisely,
take points $(x,s),(x',s')\in Y$: if $s=s'=i$ for some $i\in\matN$, then
a path $\gamma\colon [0,1]\to Y$ joining $(x,s)$ to $(x',s')$
is a geodesic if and only if $\gamma (t)=(\psi(t) , i)$  for some
geodesic $\psi\colon [0,1]\to X$ in $X$ joining $x$ to $x'$;
if $s\neq s'$, then  a path $\gamma\colon [0,1]\to Y$ joining $(x,s)$ to $(x',s')$
is a geodesic if and only if, up to reparameterization, $\gamma=\psi'\ast\varphi\ast\psi$, where
$\psi$ (respectively $\psi'$) is a (possibly constant) geodesic joining $(x,s)$ to $(p,s)$ (respectively 
$(x',s')$ to $(p,s')$), and $\varphi (t)=(p, ts'+(1-t)s)$.

Thus, let $\D\subseteq Y$ be a triangle with vertices $z_i=(x_i,s_i)$, $i=1,2,3$,
and let $l_i$ be the edge of $\D$ opposite to $z_i$.
Up to reordering, we may suppose that $\overline{z}=(\overline{x},\overline{s})\in l_1$
exists such that $\widetilde{d}(\overline{z},l_2 \cup l_3)=\delta (\D)$, and that
$s_2 \leq s_3$, whence $s_2\leq \overline{s}\leq s_3$.
If $\overline{s}\notin \matN$, then $\overline{x}=p$, and 
it is easily seen either $\overline{z}$ is a vertex of $\D$,
whence $\delta(\D)=0$, 
or $s_2<\overline{s}<s_3$. In this case
$z_2$ and $z_3$ lie in different connected components of
$Y\setminus \{\overline{z}\}$, so $\overline{z}\in l_2\cup l_3$,
and $\delta(\D)=0$ again.

So let us suppose $\overline{s}=n\in\matN$.
We set $l'_1=l_1\cap (X\times \{n\})$, and for $i=2,3$
we define $l'_i$ as follows: 
$l'_i=l_i\cap (X\times \{n\})$ if $l_i\cap (X\times \{n\})\neq\emptyset$,
and $l'_i=\{(p,n)\}$ otherwise. The previous description of the geodesics of $Y$ implies
that $l'_2\cup l'_3\subseteq l_2\cup l_3$, and that $\D'=l'_1\cup l'_2\cup l'_3$
is a geodesic triangle in $Y$ with vertices $z'_1,z'_2,z'_3$, where
$z'_i=z_i$ if $s_i=n$, $z'_i=(p,n)$ otherwise. Moreover $\D'$
is contained in $X\times \{n\}$, so it is 
the rescaled copy of 
a triangle $\D''$ in $(X,d)$. Thus  
$$
\frac {\delta (\D)}{\sz (\D)}=\frac{d(\overline{z},l_2\cup l_3)}{\sz(\D)}
\leq \frac{d(\overline{z}, l'_1\cup l'_2)}{\sz(\D')}\leq
\frac{\delta(\D')}{\sz(\D')}=\frac{\delta(\D'')}{\sz(\D'')}\leq \alpha.
$$
We have thus proved that
$$
\sup_t \frac{\Omega_Y (t)}{t}=\alpha \leq \varepsilon_H,
$$
whence in particular $\limsup_{t\to\infty} \Omega_Y (t)/t< \varepsilon_H$.
By the very definition of $\varepsilon_H$, this implies that
$Y$ is hyperbolic.

Now let $\omega$ be a ultrafilter, and consider the asymptotic cone
$Y_\omega = (Y_\omega, ((p,n)), (n))$. 
By Theorem~\ref{main2:teo}, $Y_\omega$ is a real tree.
Let us consider the map $\psi\colon X\to Y_\omega$ defined
by $\psi (x)=[(x,n)]$. It is easily seen that $\psi$ is a well-defined
isometric embedding. Since $X$ is geodesic, this readily implies that
$X$ is itself a real tree, whence the conclusion.
\finedimo

\begin{rem}
Let $Y$ be a geodesic space with $\limsup_{t\to\infty} \Omega_Y (t)/t=\alpha$.
A geodesic $\gamma_\omega\colon [0,1]\to Y_\omega$ joining 
$x_\omega=[(x_n)], y_\omega=[(y_n)]$ is called \emph{good} if it is
the $\omega$-limit of geodesics in $X$ joining $x_n$ to $y_n$, \emph{i.e.}~if
there exist geodesics $\gamma_n\colon [0,1]\to X$ such that $\gamma_\omega (t)=
[(\gamma_n (t))]$ for every $t\in [0,1]$. A slight modification of the argument
showing that any asymptotic cone of a hyperbolic space is uniquely geodesic
(see Proposition~\ref{main3:prop}, $(1)\Rightarrow (2)$) proves that
if $\gamma'$ is any geodesic in $Y_\omega$ of length $\ell$, then 
a good geodesic $\gamma$ in $Y$ exists which has the same endpoints of $\gamma'$
and is such that $d_\omega (\gamma(t),\gamma' (t))\leq 4\alpha\ell$ for every
$t\in [0,1]$. Now, it is readily seen that if $\D\subseteq Y_\omega$ is a triangle
with sides given by good geodesics, then $\delta (\D)\leq \alpha \sz(\D)$.
These facts imply that $\sup_t (\Omega_{Y_\omega} (t)/t)\leq 5\alpha$.
By Proposition~\ref{main3:prop}, this implies in turn $\varepsilon_
T\leq 5\varepsilon_H$. Note however that this inequality does not give any
information, since we already know that $1/32\leq \varepsilon_H\leq \varepsilon_T<\eta_0$,
and $32\eta_0\approx 4.8<5$.
\end{rem}




\bibliographystyle{amsalpha}
\bibliography{biblio}

\end{document}